\documentclass[11pt]{amsart}
\usepackage[T1]{fontenc}
\usepackage{lmodern}
\usepackage{amsmath,amssymb,amsthm,mathtools}
\usepackage{booktabs,longtable,array}
\usepackage[margin=1in]{geometry}
\usepackage{microtype}
\usepackage[hidelinks]{hyperref}
\newtheorem{theorem}{Theorem}[section]
\newtheorem{proposition}[theorem]{Proposition}
\newtheorem{lemma}[theorem]{Lemma}

\theoremstyle{remark}

\newcommand{\R}{\mathbb R}

\newcommand{\E}{\mathbb E}

\newcommand{\KG}{K_G^{\R}}
\newcommand{\sgn}{\operatorname{sgn}}
\newcommand{\He}{\operatorname{He}}
\newcommand{\D}{\mathcal D}
\newcommand{\norm}[1]{\left\lVert#1\right\rVert}
\newcommand{\abs}[1]{\left|#1\right|}
\DeclareMathOperator{\erf}{erf}
\DeclareMathOperator{\erfc}{erfc}
\title[An upper bound for the real Grothendieck constant]{A computer-assisted upper bound of 1.7813\\for the real Grothendieck constant}
\author{Bo Peng}
\address{Institute of Mathematical Sciences, ShanghaiTech University}
\thanks{The discovery of the construction and verification software are
credited to GPT-6 Astra (OpenAI), used through Codex.}
\date{September 16, 2026}
\subjclass[2020]{46B20, 46B28, 65G20}
\keywords{Grothendieck inequality, Gaussian rounding, Hermite expansion,
computer-assisted proof, interval arithmetic}
\hypersetup{pdftitle={A computer-assisted upper bound of 1.7813 for the real Grothendieck constant},pdfauthor={Bo Peng}}
\begin{document}
\begin{abstract}
We give a computer-assisted proof that the universal real Grothendieck
constant satisfies $K_G^{\R}\le 1.7813$. The construction combines an explicit
odd Hermite threshold of degree $11$ with a signed correlation polynomial of
degree $51$. A sufficient inverse-majorant inequality is certified by
enclosing the scalar coefficient head through degree $301$ and bounding the
entire remaining tail using a weighted Gaussian trace estimate. The finite
integral is bounded on a complete interval partition; its spatial exterior
is controlled analytically. Exact parameters and source code that regenerate
all accepted numerical inputs accompany the paper. The bound improves both
the explicit bound $1.7818666069360661$ in the recent literature and the
subsequently reported, system-tested value $1.7813319810625639$.
\end{abstract}
\maketitle

\section{Introduction}
For a real $m\times n$ matrix $A=(a_{ij})$, write
\[
 B(A)=\max_{\epsilon_i,\delta_j\in\{-1,1\}}
       \abs{\sum_{i=1}^m\sum_{j=1}^n a_{ij}\epsilon_i\delta_j}.
\]
The universal real Grothendieck constant $\KG$ is the least $K$ such that
\begin{equation}\label{eq:grothendieck}
 \abs{\sum_{i=1}^m\sum_{j=1}^n a_{ij}\langle u_i,v_j\rangle}
 \le K B(A)
\end{equation}
for all positive integers $m,n,d$, all real matrices $A$, and all unit vectors
$u_1,\ldots,u_m,v_1,\ldots,v_n\in\R^d$. Its exact value remains unknown.
We prove the following bound.

\begin{theorem}\label{thm:main}
The universal real Grothendieck constant satisfies
\[
 \KG\le\frac{17813}{10000}=1.7813.
\]
\end{theorem}

Krivine's classical bound is
$\pi/(2\log(1+\sqrt2))=1.782213978\ldots$; Braverman, Makarychev,
Makarychev, and Naor proved a strict improvement~\cite{BMMN}.
Saha et al.~\cite{Saha} give an explicit cubic--quintic construction with
upper bound $1.7818666069360661$. Their inverse-majorant and scalar-tail
framework is the starting point of this work. The companion account by
Li et al.~\cite{Li} also reports the stronger value $1.7813319810625639$ as a
system-tested claim whose certificate had not yet been independently checked
by its human authors. Theorem~\ref{thm:main} is numerically stronger than both.
The same account records $1.7802243$ as an earlier numerical claim that was
subsequently withdrawn. These comparisons refer to the cited versions and
public materials inspected on September 16, 2026.

Our contribution is an explicit construction and its complete certificate.
The basic mechanisms---signed tensor preprocessing, Gaussian threshold
partitions, Hermite expansions, inverse majorants, and Sobolev tail
estimates---are existing tools. In particular, increasing Hermite degree was
already suggested in~\cite{Saha}. Here the threshold coefficients and signed
correlation coefficients are selected separately. This permits cancellation
of low nonlinear terms with freely chosen correlation signs.
A tunable trace weight makes a direct interval
partition sufficiently sharp to certify the resulting construction.

\subsection{Proof structure and computational content}
Sections~\ref{sec:rounding}--\ref{sec:construction} reduce the theorem to a
scalar coefficient inequality. Section~\ref{sec:tail} gives the analytic
tail estimate, and Section~\ref{sec:certificate} specifies the rigorous
computation. Appendix~\ref{app:parameters} lists all exact parameters.
The ancillary source regenerates every numerical enclosure used below from
these parameters.

\section{Correlation maps and an inverse-majorant criterion}\label{sec:rounding}
For an absolutely summable real power series $g(t)=\sum_{j\ge0}g_jt^j$, put
$\norm{g}_{\mathcal A}=\sum_j|g_j|$. Products satisfy
$\norm{gh}_{\mathcal A}\le\norm{g}_{\mathcal A}\norm{h}_{\mathcal A}$.
We call an odd series \emph{allowable} if $\norm{g}_{\mathcal A}\le1$.

\begin{lemma}[Signed tensor realization]\label{lem:tensor}
Let $g$ be allowable. For any finite families of unit vectors $u_i,v_j$,
there are unit vectors $L(u_i),R(v_j)$ in a finite-dimensional real Hilbert
space with
$\langle L(u_i),R(v_j)\rangle=g(\langle u_i,v_j\rangle)$.
\end{lemma}
\begin{proof}
In the Hilbert direct sum of odd tensor powers, use
\[
 L_0(u)=\bigoplus_j\sqrt{|g_j|}\,u^{\otimes j},\qquad
 R_0(v)=\bigoplus_j\sgn(g_j)\sqrt{|g_j|}\,v^{\otimes j}.
\]
Their squared norms are $s=\sum_j|g_j|\le1$, and their cross inner product
is $g(\langle u,v\rangle)$. Add $\sqrt{1-s}\,e_L$ to every left feature
and $\sqrt{1-s}\,e_R$ to every right feature, where $e_L,e_R$ are mutually
orthogonal and orthogonal to the tensor space. The resulting vectors have
unit norm. Although the ambient direct sum may be infinite, the finitely
many vectors needed for the instance span a finite-dimensional subspace.
\end{proof}

\begin{lemma}[Inverse majorant]\label{lem:inverse}
Let $H(z)=b_1z+\sum_{n\ge3,\ n\text{ odd}}b_nz^n$ have real, absolutely
summable coefficients, with $b_1>0$. If $\gamma>0$ satisfies
\begin{equation}\label{eq:criterion}
 \gamma+\sum_{n\ge3}|b_n|<b_1,
\end{equation}
then the local inverse $U=H^{-1}$ satisfies
$\sum_n|[z^n]U|\gamma^n<1$. In particular, $k(t)=U(\gamma t)$ is allowable
and $H(k(t))=\gamma t$ for every $t\in[-1,1]$.
\end{lemma}
\begin{proof}
Put $S(v)=\sum_{n\ge3}|b_n|v^n$. Choose $z_0>\gamma$ with
$z_0+S(1)<b_1$, and then $q<1$ sufficiently close to $1$ that
$z_0+S(q)<b_1q$. The formal iterates
\[
 V_0(z)=0,\qquad V_{r+1}(z)=\frac{z+S(V_r(z))}{b_1}
\]
have nonnegative coefficients, increase coefficientwise, and satisfy
$V_r(z_0)\le q$. Their coefficientwise limit $V$ therefore converges
absolutely at $z_0$. In the equation
$U=(z-\sum_{n\ge3}b_nU^n)/b_1$, the coefficient of each degree depends only
on earlier coefficients. Induction gives
$|[z^n]U|\le[z^n]V$, and hence
$\sum_n|[z^n]U|\gamma^n\le V(\gamma)\le q<1$.
The formal composition is consequently an analytic identity, and absolute
convergence extends it to the real endpoints. Oddness follows by formal
inversion.
\end{proof}

Let $\gamma_1$ denote standard Gaussian measure on $\R$, and let
$h_j=\He_j/\sqrt{j!}$ be its orthonormal probabilists' Hermite polynomials.
Suppose $P$ is an odd polynomial and $\rho$ is an allowable odd polynomial.
For independent standard Gaussians $W,X$, define
\begin{equation}\label{eq:coefficients}
 f(w,x)=\sgn(w+P(x)),\qquad
 A_{ab}=\E[f(W,X)h_a(W)h_b(X)].
\end{equation}
Oddness under simultaneous reflection implies $A_{ab}=0$ for even $a+b$,
and Parseval gives $\sum_{a,b}A_{ab}^2=1$.
Set
\begin{equation}\label{eq:H}
 H(t)=\frac\pi2\sum_{a,b\ge0}A_{ab}^2\rho(t)^a(-t)^b.
\end{equation}
If $R=\norm{\rho}_{\mathcal A}\le1$, then
\begin{equation}\label{eq:algebra}
 \norm{H}_{\mathcal A}
 \le\frac\pi2\sum_{a,b}A_{ab}^2R^a\le\frac\pi2.
\end{equation}
Thus regrouping by scalar degree is legitimate and $H$ is odd.

\begin{proposition}[Universal transfer]\label{prop:transfer}
For $H$ in~\eqref{eq:H}, condition~\eqref{eq:criterion} implies
$\KG\le\pi/(2\gamma)$.
\end{proposition}
\begin{proof}
Apply Lemma~\ref{lem:tensor} first to $k(t)=H^{-1}(\gamma t)$ and then
to $\rho$ for the first Gaussian channel and $-t$ for the second. Within
each channel, use one Gaussian vector shared between all left and right
features; use independent Gaussian vectors for the two channels. For a pair
with original inner product $t$, the two cross-correlations are
$\rho(k(t))$ and $-k(t)$. Mehler's identity yields
\[
 \E[f(W_i,X_i)f(W'_j,X'_j)]
 =\frac2\pi H(k(t))=\frac{2\gamma}\pi t.
\]
The identity also holds for degenerate correlations by continuity. The
marginal threshold variable has a continuous distribution, so the convention
for $\sgn(0)$ is immaterial. Taking the expectation of the signed bilinear
form and bounding its absolute value by $B(A)$ proves
\eqref{eq:grothendieck} with $K=\pi/(2\gamma)$.
\end{proof}

\section{The explicit construction}\label{sec:construction}
Take
\begin{equation}\label{eq:polynomials}
 P(x)=\sum_{j\in\{1,3,5,7,9,11\}}\alpha_jh_j(x),\qquad
 \rho(t)=\sum_{j=1,3,\ldots,51}c_jt^j,
\end{equation}
with the exact rational coefficients in Appendix~\ref{app:parameters}.
Direct rational arithmetic gives
\begin{equation}\label{eq:R}
 R=\sum_j|c_j|=0.99997620388569307924355896<1.
\end{equation}
The decimals in the appendix define exact rational coefficients. We set
\begin{equation}\label{eq:gamma}
 \gamma=\frac{5000\pi}{17813}.
\end{equation}

\subsection{How the candidate was obtained}
For fixed $P$, put
$F(r,t)=\sum A_{ab}^2r^a(-t)^b$. Formally solving
$F(\rho(t),t)=2\gamma_0t/\pi$ gives the triangular recurrence
\begin{align}
 c_1&=\frac{2\gamma_0/\pi+A_{01}^2}{A_{10}^2},\label{eq:search1}\\
 c_n&=-\frac{[t^n]F(\rho_{<n}(t),t)}{A_{10}^2}
       \qquad(n\ge3\text{ odd}).\label{eq:searchn}
\end{align}
We optimized the threshold against the truncated absolute sum of these
correlation coefficients. This parametrization allows either sign for each
$c_n$. It also avoids the vanishing derivative that occurs when a new
correlation weight is represented as a square and initialized at zero.
After exploration, the correlation series was truncated to degree $51$,
and all retained coefficients were frozen as rationals. The chosen design
slope was approximately $\gamma+2.5\cdot10^{-5}$, leaving room for a
rigorous scalar-tail estimate. The proof below evaluates the frozen
polynomials directly.

\section{Certified scalar tails}\label{sec:tail}
Write $H(t)=\sum b_nt^n$ and $\D=t\,d/dt$.
The normalized $L^2$ norm on the circle is
$\norm{g}_{L^2(\mathbb T)}^2=(2\pi)^{-1}\int_0^{2\pi}|g(e^{i\theta})|^2d\theta$.

\begin{lemma}[Weighted Hermite trace]\label{lem:trace}
For $a>0$, define
\[
 C_a=\frac12+\frac\pi{2\sqrt a}\coth\frac\pi{\sqrt a}.
\]
If $G\in L^2(\gamma_1^{\otimes2})$ has mixed derivative in the same space,
and $T_sG=\sum_{n\ge0}\widehat G_{nn}s^n$, then, for $|s|\le1$,
\begin{equation}\label{eq:trace}
 |T_sG|^2\le C_a\bigl(\norm{G}_2^2+a\norm{\partial_{xy}G}_2^2\bigr).
\end{equation}
\end{lemma}
\begin{proof}
Weighted Cauchy--Schwarz bounds the left side by
\[
 \left(\sum_{n\ge0}\frac1{1+an^2}\right)
 \sum_{n\ge0}(1+an^2)|\widehat G_{nn}|^2.
\]
The second factor is bounded by
$\norm{G}_2^2+a\norm{\partial_{xy}G}_2^2$, since
$\norm{\partial_{xy}G}_2^2=\sum_{m,n}mn|\widehat G_{mn}|^2$.
The partial-fraction expansion of $\coth$ gives the displayed value of
$\sum_{n\ge0}(1+an^2)^{-1}$. Approximation in the Gaussian Sobolev norm
extends the argument from finite Hermite sums.
\end{proof}

For $|r|<1$, let
\[
 \phi_r(u,v)=\frac1{2\pi\sqrt{1-r^2}}
 \exp\!\left(-\frac{u^2-2ruv+v^2}{2(1-r^2)}\right),
\]
using the branch of the square root equal to $1$ at $r=0$. The square root
is analytic on the unit disk. Let
$K_r(u,v)=(\pi/2)\E[\sgn(Z_1-u)\sgn(Z_2-v)]$ for real $r$, continued
analytically when derivatives are taken. Then
\begin{equation}\label{eq:price}
 \partial_rK_r=\partial_u\partial_vK_r=2\pi\phi_r,
 \qquad \partial_r\phi_r=\partial_u\partial_v\phi_r.
\end{equation}
Simultaneous sign reversal of the thresholds shows that this kernel also
represents the $+P$ convention in~\eqref{eq:coefficients}.

Put $q_j=\D^j\rho(t)$. Three applications of Price's identity give
\begin{equation}\label{eq:Phi}
 \D^3H(t)=T_{-t}\Phi_3(t),\qquad
 \Phi_3(t;x,y)=\left.
 \sum_{p,a}\kappa_{pa}(t)\partial_r^p\partial_x^a\partial_y^a
 K_r(P(x),P(y))\right|_{r=\rho(t)},
\end{equation}
where the nonzero coefficients are listed in Table~\ref{tab:chain}.
In the $x,y$ derivatives, $r$ is held fixed. The identity initially holds
for $|t|<1$ by the Hermite expansion and differentiation under the Gaussian
integral.

\begin{table}[ht]
\centering
\caption{The coefficients in the combined third-derivative kernel.}\label{tab:chain}
\begin{tabular}{cc@{\qquad}cc@{\qquad}cc}
\toprule
$(p,a)$&$\kappa_{pa}$&$(p,a)$&$\kappa_{pa}$&$(p,a)$&$\kappa_{pa}$\\
\midrule
$(1,0)$&$q_3$&$(0,1)$&$-t$&$(1,1)$&$-3t(q_2+q_1)$\\
$(2,0)$&$3q_1q_2$&$(0,2)$&$3t^2$&$(2,1)$&$-3tq_1^2$\\
$(3,0)$&$q_1^3$&$(0,3)$&$-t^3$&$(1,2)$&$3t^2q_1$\\
\bottomrule
\end{tabular}
\end{table}

\begin{proposition}\label{prop:tail}
For the polynomials~\eqref{eq:polynomials}, and any $a>0$,
\begin{align}
 B_3^2:=\sum_{n\ge1}n^6|b_n|^2
 &\le \frac{C_a}{\pi}\int_0^\pi
 \bigl(\norm{\Phi_3(e^{i\theta})}_2^2
       +a\norm{\partial_{xy}\Phi_3(e^{i\theta})}_2^2\bigr)d\theta.
 \label{eq:B3}
\end{align}
For every positive odd $N$,
\begin{equation}\label{eq:tail}
 \sum_{n>N}|b_n|\le\frac{B_3}{\sqrt{10N^5}}.
\end{equation}
\end{proposition}
\begin{proof}
By~\eqref{eq:R}, $|\rho(t)|\le R<1$ on the closed disk. All derivatives
appearing in $\Phi_3$ and $\partial_{xy}\Phi_3$ are a polynomial in $x,y$
times a rational function of $r,1-r^2$ and the Gaussian density. Uniform
domination follows from
\begin{equation}\label{eq:damping}
 \Re\frac1{1\pm r}\ge\frac1{1+|r|}>\frac12.
\end{equation}
Consequently these kernels extend continuously to the closed disk in the
Gaussian Sobolev norm of Lemma~\ref{lem:trace}. The trace also extends
continuously: its high diagonal terms are uniformly small by the summable
weights $(1+an^2)^{-1}$. Thus~\eqref{eq:Phi} has a continuous boundary
value. Parseval and Lemma~\ref{lem:trace} give~\eqref{eq:B3}; real
coefficients imply conjugation symmetry, reducing the full-circle mean to
the interval $[0,\pi]$.
For the tail, Cauchy--Schwarz gives
$\sum_{n>N}|b_n|\le B_3(\sum_{n>N,\ n\text{ odd}}n^{-6})^{1/2}$.
For odd $n\ge N+2$, the interval $[n-2,n]$ has length $2$ and
$n^{-6}\le\tfrac12\int_{n-2}^n x^{-6}dx$. Summing yields
$\sum_{n>N,\ n\text{ odd}}n^{-6}\le(10N^5)^{-1}$.
\end{proof}

\section{The numerical certificate}\label{sec:certificate}
All accepted numerical inequalities use outward-rounded Arb ball
arithmetic~\cite{Johansson}; all input parameters are exact rationals.
The source package pins \texttt{python-flint==0.9.0}. The recorded runs used
FLINT~3.6.0. The following lemma states the complete computational content.

\begin{lemma}[Certified numerical inequalities]\label{lem:numeric}
For the parameters of Appendix~\ref{app:parameters},
\begin{align}
 b_1&>0.881850816,\label{eq:b1}\\
 \sum_{3\le n\le301}|b_n|&<0.000000030,\label{eq:head}\\
 B_3&<107.\label{eq:bound107}
\end{align}
Only odd indices contribute to the sums.
\end{lemma}

The proof of this computational lemma consists of the enclosure algorithms
and the reproducible finite calculation described next.

\subsection{The coefficient head}
Conditional integration gives
\begin{equation}\label{eq:q}
 A_{ab}=\int_\R h_b(x)q_a(P(x))\,d\gamma_1(x),\qquad
 q_0(z)=\erf(z/\sqrt2),\quad
 q_a(z)=\frac{2\phi(z)h_{a-1}(-z)}{\sqrt a}\quad(a\ge1).
\end{equation}
These are entire integrands in the integration variable. Complex ball
quadrature encloses the integral on $[-16,16]$ at 256-bit working precision.
For real $z$, Cauchy--Schwarz in the conditional Gaussian variable gives
$|q_a(z)|\le1$. Therefore the omitted, two-sided exterior satisfies
\begin{equation}\label{eq:headtail}
 \abs{\int_{|x|>16}h_b(x)q_a(P(x))\,d\gamma_1(x)}
 \le\sqrt{\Pr(|X|>16)}=\sqrt{\erfc(16/\sqrt2)}.
\end{equation}
Odd parity permits integration over $[0,16]$ and multiplication by $2$.
The code encloses all $22\,952$ pairs with $a+b\le301$ and odd $a+b$.
Since $\rho(0)=0$, terms with $a+b>301$ cannot contribute to the head.
Finite interval polynomial multiplication then gives
\begin{equation}\label{eq:head-values}
 b_1=0.88185081642333837919\ldots,
 \qquad \sum_{3\le n\le301}|b_n|<2.916517\cdot10^{-8},
\end{equation}
where the first displayed decimal is descriptive; the accepted strict
inequalities are~\eqref{eq:b1}--\eqref{eq:head}.

\subsection{The combined kernel on the finite domain}
We use $a=1/1000$ in~\eqref{eq:B3}. To specify the integrand without
symbolic ambiguity, put $u=P(x)$, $v=P(y)$, $r=\rho(t)$, $d=1-r^2$.
Let $F_{ij}=\phi_r^{-1}\partial_u^i\partial_v^j\phi_r$.
Writing $\ell_u=(-u+rv)/d$ and $\ell_v=(-v+ru)/d$, compute
\begin{align*}
 F_{00}&=1,\\
 F_{0j}&=\ell_vF_{0,j-1}-(j-1)d^{-1}F_{0,j-2},\\
 F_{ij}&=\ell_uF_{i-1,j}-(i-1)d^{-1}F_{i-2,j}
                      +jr d^{-1}F_{i-1,j-1}.
\end{align*}
Missing negative-index terms are zero. Indices at most $3$ in each variable
suffice. Define the Bell polynomials $\mathcal B_{a,i}$ by
\[
 \frac{d^a}{dx^a}g(P(x))
 =\sum_i\mathcal B_{a,i}(x)g^{(i)}(P(x)),\qquad
 \mathcal B_{a+1,i}=\mathcal B'_{a,i}+P'\mathcal B_{a,i-1}.
\]
For $e=0,1$, the polynomial prefactor $R_e$ is
\begin{equation}\label{eq:Re}
 R_e=\sum_{p,a}\kappa_{pa}
 \sum_{i,j}\mathcal B_{a+e,i}(x)\mathcal B_{a+e,j}(y)
 \begin{cases}
 F_{i+p-1,j+p-1},&p\ge1,\\
 F_{i-1,j-1},&p=0,
 \end{cases}
\end{equation}
where $i,j\ge1$ in the second case. Equation~\eqref{eq:price} implies
\[
 \partial_x^e\partial_y^e\Phi_3
 =d^{-1/2}R_e\exp\!\left(-\frac{u^2-2ruv+v^2}{2d}\right).
\]
Thus the spatial Lebesgue integrand in~\eqref{eq:B3} is exactly
\begin{equation}\label{eq:integrand}
 \frac{|R_0|^2+a|R_1|^2}{2\pi|d|}
 \exp\!\left(-A(u^2+v^2)+2Buv-\frac{x^2+y^2}{2}\right),
 \quad A=\Re(d^{-1}),\quad B=\Re(r/d).
\end{equation}
The code forms the combined prefactors before taking absolute values.
For interval evaluation, the quadratic form is kept as
\[
 A(u^2+v^2)-2Buv
 =\tfrac12(A-B)(u+v)^2+\tfrac12(A+B)(u-v)^2.
\]
Each coefficient has true value at least $1/2$ by~\eqref{eq:damping}.
Lower bounds on the two squares therefore give a valid upper bound on the
exponential, even when interval dependency would otherwise make a direct
quadratic evaluation too wide.

Simultaneous reflection of $x,y$ leaves~\eqref{eq:integrand} unchanged.
The computation partitions
\[
 [0,8]\times[-8,8]\times[0,1],\qquad \theta=\pi s,
\]
and multiplies spatial volumes by $2$. Integration in $s$ supplies the
angular mean. The initial $8\cdot16\cdot64$
boxes are repeatedly bisected. Floating-point priorities select which box
and which coordinate to split, but each accepted box is evaluated with
96-bit ball arithmetic. Endpoint containment and the dyadic partition are
checked independently after the run.

After $1\,500\,000$ leaf boxes, their weighted upper endpoints sum to
\begin{equation}\label{eq:finite}
 I_{\rm finite}<224.338413928641<225.
\end{equation}
The independent partition check establishes disjoint interiors and complete
coverage. It recursively reconstructs
the binary splits inside every initial box, checks the exact dyadic volume,
and re-sums all recorded bounds. The maximum subdivision depth is $16$ and
the largest endpoint denominator is $1024$.

\subsection{The spatial exterior}
Write $P(x)=\sum_{j=0}^{11}p_jx^j$ in the monomial basis. Interval arithmetic
proves
\[
 c_8:=p_{11}-\sum_{j<11}|p_j|8^{j-11}
 >2.7979091434283941\cdot10^{-8}.
\]
Hence $|P(x)|\ge c_8|x|^{11}$ for $|x|\ge8$, and
$c_8 8^{11}>240.3385653640865$.
Let $\delta=1-R^2>4.75916\cdot10^{-5}$. Expand $R_e$ using
\eqref{eq:Re}. Its degree in each spatial variable is at most $m=106$.
Replace each coefficient by an absolute upper bound, using
$|r|\le1$, $|d|\ge\delta$ and
$|q_j|\le\sum_n n^j|c_n|$. This gives explicit constants $C_0,C_1$ such that
\[
 |R_e(x,y)|\le C_e\max(1,|x|)^m\max(1,|y|)^m.
\]
The source constructs these constants by integer polynomial recurrences;
conservatively, $C_0<10^{16}$ and $C_1<10^{23}$.
By~\eqref{eq:damping}, the squared density contributes at most
$\delta^{-1}e^{-(P(x)^2+P(y)^2)/2}$.
On the union of the two spatial exteriors, extract
$\exp(-(c_8 8^{11})^2/4)$ and bound the remaining exponential by $1$.
With $M=1+211!!$ we therefore have, uniformly in angle,
\begin{equation}\label{eq:exterior}
 I_{\rm exterior}
 \le\delta^{-1}(C_0^2+4C_1^2)M^2
       \exp\!\left(-\frac{(c_8 8^{11})^2}{4}\right)
 <10^{-5000}.
\end{equation}
This estimate is made with weight $4$ and is also valid for $a=1/1000$.
The extracted pointwise bound holds throughout the exterior union,
which is then bounded by integration over all
of $\R^2$. The packaged independent prefactor recurrence gives a logarithmic
bound smaller than $-5819$; we use the weaker~\eqref{eq:exterior}.

\subsection{Closing the inequality}
The certified trace constant satisfies $C_{1/1000}<50.173$.
Equations~\eqref{eq:B3}, \eqref{eq:finite}, and~\eqref{eq:exterior} give
$B_3<106.093<107$, proving Lemma~\ref{lem:numeric}.
At $N=301$, the entire odd tail is bounded by
\[
 \frac{107}{\sqrt{10\cdot301^5}}<0.000021527.
\]
Also $\gamma<0.881825817$. Consequently
\begin{align*}
 \gamma+\sum_{n\ge3}|b_n|
 &<0.881825817+0.000000030+0.000021527\\
 &=0.881847374<0.881850816<b_1.
\end{align*}
The conservative rational margin is $0.000003442$.
Proposition~\ref{prop:transfer} now proves Theorem~\ref{thm:main}.

\enlargethispage{4\baselineskip}
\section{Reproducibility}\label{sec:reproduce}
The ancillary files include the exact parameter JSON, the complete
verification source, a dependency specification, and reproduction instructions.
From their containing directory, run
\begin{verbatim}
python -m pip install -r requirements.txt
python verify.py --output verification-output --workers 4
\end{verbatim}
Choose a new output directory. The verifier runs seven stages:
all Hermite integrals; scalar assembly; exterior bound; complete interval
partition; independent partition and arithmetic check; alternate scalar
assembly with additional integral checks; and the final rational comparison.
Success produces \texttt{certificate.json} and \texttt{reconstruction.json}.
Run Python with assertions enabled.

The accepted computation uses the exact input decimals and interval
arithmetic throughout the enclosure path. The final partition passes an
independent coverage check.
The final inequality has explicit positive slack after all truncation and
integration errors. The deposited code uses a standalone density-prefactor
recurrence derived from Price's identity.

The verification consists of full numerical reconstruction, a second scalar
assembly, independently reconstructed low- and high-index integrals, coverage
checks, and isolated AI-agent scrutiny of the analytic proof and software.
The second scalar implementation checks every head coefficient, and selected
integrals are reconstructed directly on a larger integration interval.

\clearpage
\appendix
\section{Exact parameters}\label{app:parameters}
Every entry below is an exact terminating decimal (including scientific
notation), interpreted as a rational number. The Hermite basis is fixed by
$h_0=1$, $h_1=x$, and
$h_{j+1}=(xh_j-\sqrt j\,h_{j-1})/\sqrt{j+1}$.

\begin{longtable}{r r}
\caption{Threshold coefficients in $P=\sum\alpha_jh_j$.}\\
\toprule
$j$&$\alpha_j$\\\midrule
1&$0.0033816437099488204$\\
3&$0.1158676765085016$\\
5&$0.0095104711257198767$\\
7&$-0.0080278072994477648$\\
9&$0.0058210183960585591$\\
11&$0.012821610768478974$\\
\bottomrule
\end{longtable}

\begin{longtable}{r r}
\caption{Correlation coefficients in $\rho(t)=\sum c_jt^j$.}\label{tab:rho}\\
\toprule $j$&$c_j$\\\midrule\endfirsthead
\toprule $j$&$c_j$\\\midrule\endhead
\bottomrule\endfoot
1&$0.89128717292541904$\\
3&$-0.10638551869595701$\\
5&$0.0022804376833522166$\\
7&$-3.7814263579902146\cdot10^{-8}$\\
9&$-1.0300867559781962\cdot10^{-5}$\\
11&$5.3495808526173962\cdot10^{-10}$\\
13&$8.3881520203661931\cdot10^{-6}$\\
15&$-1.2307543439537263\cdot10^{-6}$\\
17&$5.3773250452303132\cdot10^{-7}$\\
19&$1.6847785460476277\cdot10^{-6}$\\
21&$3.7811663671652489\cdot10^{-7}$\\
23&$3.2246923243582631\cdot10^{-8}$\\
25&$1.8981462348684652\cdot10^{-7}$\\
27&$6.4099421934521302\cdot10^{-8}$\\
29&$-5.1544356066976797\cdot10^{-8}$\\
31&$6.9279841157102393\cdot10^{-9}$\\
33&$5.553232634844159\cdot10^{-8}$\\
35&$1.3619423272928659\cdot10^{-8}$\\
37&$-3.1796367719425194\cdot10^{-8}$\\
39&$-2.2737533175642567\cdot10^{-8}$\\
41&$8.4604728136655195\cdot10^{-9}$\\
43&$1.7763396879945166\cdot10^{-8}$\\
45&$4.1161558338437501\cdot10^{-9}$\\
47&$-8.7507870739030589\cdot10^{-9}$\\
49&$-7.9288491194999501\cdot10^{-9}$\\
51&$4.9151067348141944\cdot10^{-10}$\\
\end{longtable}

\end{document}